\documentclass[12pt]{article}

\newcommand{\version}{\today}

\usepackage{pictex}
\newcommand{\wv}{\beginpicture
    \setcoordinatesystem units <0.8 mm,0.8 mm>
    \setplotarea x from  -1 to 1, y from  -1 to 1
    \put{\textcolor{white}{\circle*{4}}} [B1] at 2.0 0
    \put{\textcolor{black}{\circle{5}}} [B1] at 2.2 0
    \endpicture}

\usepackage{color}

\usepackage{amsmath,amssymb,geometry,enumerate,float}
\usepackage{mathtools}
\usepackage{amsthm}

\newtheorem{theorem}{Theorem}[section]
\newtheorem{lemma}[theorem]{Lemma}

\newtheorem{conjecture}[theorem]{Conjecture}
\newtheorem{observ}[theorem]{Observation}

\newcommand{\claw}{K_{1,3}}
\newcommand{\mf}{\mathcal{F}}

\begin{document}

\footnotetext[1]{{Institut f{\"u}r Diskrete Mathematik und Algebra,
  Technische Universit{\"a}t Bergakademie Freiberg, 09596 Freiberg, Germany}; \\ email: {\tt brause@math.tu-freiberg.de, Ingo.Schiermeyer@tu-freiberg.de }}
\footnotetext[2]{Department of Mathematics, University of
West Bohemia;
European Centre of Excellence NTIS -
New Technologies for the Information Society;
P.O. Box 314, 306 14
Pilsen, Czech Republic; \\ email: {\tt $\{$holubpre,vranap,ryjacek$\}$@kma.zcu.cz, kabela@ntis.zcu.cz}}


\title{On forbidden induced subgraphs for $\claw$-free perfect graphs\footnote{Research partly supported by the DAAD-PPP project
``Colourings and connection in graphs" with project-ID 57210296 (German) and 7AMB16DE001 (Czech), respectively.}
}

\author{Christoph Brause$^{1}$ \and P\v remysl Holub$^{2,}$\footnote{Research partly supported by project P202/12/G061 of the
Czech Science Foundation.} \and Adam Kabela$^{2,\dagger}$ \and Zden\v{e}k Ryj{\'a}\v{c}ek$^{2,\dagger}$  \and Ingo Schiermeyer$^{1}$ \and Petr Vr{\'a}na$^{2,\dagger}$ \\}

\date{\version}

\maketitle

\begin{abstract}
Considering connected
$\claw$-free graphs with
independence number at least $3$,
Chudnovsky and Seymour (2010) 
showed that every such graph, say $G$, is $2\omega$-colourable
where $\omega$ denotes the clique number of $G$.
We study $(\claw,Y)$-free graphs, 
and show that the following three statements are equivalent.
\begin{enumerate}
\item[$(1)$] Every connected $(\claw,Y)$-free graph which is distinct from an odd cycle and which has independence number at least $3$
     is perfect.
\item[$(2)$] Every connected $(\claw,Y)$-free graph which is distinct from an odd cycle and which has independence number at least $3$
     is $\omega$-colourable.
\item[$(3)$] $Y$ is isomorphic 
	to an induced subgraph of $P_5$ or $Z_2$ (where $Z_2$ is also known as hammer).
\end{enumerate} 
Furthermore, for connected $(\claw,Y)$-free graphs
(without an assumption on the independence number), we show a similar characterisation 
featuring the graphs $P_4$ and $Z_1$ (where $Z_1$ is also known as paw).

\smallskip


\end{abstract}

\section{Introduction}
We consider finite, simple, and undirected graphs but do not distinguish between isomorphic graphs. 
For terminology and notation not defined here, we refer the reader to~\cite{BM08}. 
 
We recall that an {\it induced subgraph} of a graph $G$
is a graph on a vertex set $S \subseteq V(G)$
for which two vertices are adjacent
if and only if these two are adjacent in $G$;
in particular, we say that this subgraph is \emph{induced by $S$}.
Furthermore, a graph $H$ is an \emph{induced subgraph} of $G$ if $H$ is isomorphic
to an induced subgraph of $G$.
$G$ is \emph{$(H_1, H_2, \dots, H_k)$-free}
if none of its induced subgraphs is isomorphic to one of $\{H_1, H_2, \dots, H_k\}$.

A graph is {\it $k$-colourable} if its vertices can be coloured with $k$ colours
so that adjacent vertices obtain distinct colours.
The smallest integer $k$ such that a given graph $G$ is $k$-colourable
is called the {\it chromatic number} of $G$, denoted by $\chi(G).$
Clearly, $\chi(G) \geq \omega(G)$ for every graph $G,$
where $\omega(G)$, the so-called \emph{clique number} of $G$, denotes the order of a maximum complete subgraph of $G.$
For brevity, we shall say that a graph $G$ is \emph{$\omega$-colourable} if $\chi(G) = \omega(G)$.
Furthermore, $G$ is {\it perfect} if $\chi(G')=\omega(G')$
for every induced subgraph $G'$ of $G.$
%

Considering a relation between colourings and forbidden induced subgraphs,
Gy\'arf\'as and, independently, Sumner conjectured the following:

\begin{conjecture}[Gy\'arf\'as-Sumner Conjecture~\cite{G},\cite{S}]\label{conj_G}
For every forest $F$, there is a function $f_F\colon \mathbb{N}\to \mathbb{N}$
such that if $G$ is an $F$-free graph, then $G$ is $f_F(\omega(G))$-colourable.
\end{conjecture}

If the graph $F$ contains a cycle, then
no such function $f_F$ exists since Erd\H{o}s~\cite{E} showed that
there are graphs of arbitrarily large girth and chromatic number.

Considering forbidden induced forests, for example, Gy\'arf\'as proved $\chi(G)\leq R(\omega(G),3)$ for every $K_{1,3}$-free graph $G$; here $R(\omega(G),3)$ denotes the classical Ramsey number, that is, $R(\omega(G),3)$ denotes the least integer $k$ such that each graph of order at least $k$ contains an independent set of size $3$, or an induced subgraph which is complete and of order $\omega(G)$. 
This upper bound follows from Ramsey's Theorem~\cite{Ramsey}
and from the trivial inequality $\chi(G) \leq \Delta(G)+1$, where $\Delta(G)$ denotes the maximum degree of $G$.
However, $R(\omega(G),3)$ is known to be non-linear in $\omega(G)$, cf.\,\cite{K}.
Additionally, every upper bound on the chromatic number in terms of the clique number of a $K_{1,3}$-free graph is non-linear since
if $G$ is a graph with independence number~$2,$
then $\chi(G) \geq \frac{|V(G)|}{2}$ (since every colour can be used at most twice)
but $\omega(G)$ might be at most $9 \sqrt{|V(G)| \cdot \log|V(G)|}$
(by a result of Kim~\cite{K}).
In particular, Brause et al.~\cite{BRSV} showed that even
in the class of $(2K_2, 3K_1)$-free graphs, we cannot bound the chromatic number of such a graph by a linear function depending on its clique number only.
Excluding all graphs with independence number $2,$
Chudnovsky and Seymour~\cite{CS} showed that
every connected $\claw$-free graph $G$ with independence number at least $3$
is $2\omega(G)$-colourable.

In addition, Brause et al., showed the following: 

\begin{theorem}[Brause et al.~\cite{BRSV}]\label{thm_Brause_perfect}
Every connected $(\claw, 2K_2)$-free graph with independence number at least $3$
is perfect.
\end{theorem}

In this paper, we generalise the result of Theorem~\ref{thm_Brause_perfect}.
We recall that, given a graph $G$, a {\it hole} in $G$ is an induced cycle of length at least $4$,
and an {\it antihole} in $G$ is an induced subgraph whose complement is a cycle of length at least $4$.
A hole (antihole) is {\it odd} if it has an odd number of vertices.
As a main tool, we shall use the following result of Chudnovsky et al.~\cite{CRST}.

\begin{theorem}[The Strong Perfect Graph Theorem~\cite{CRST}]\label{SPGT}
A graph is perfect if and only if it contains neither an odd hole
nor an odd antihole.
\end{theorem}

With the aid of Theorem~\ref{SPGT} and using results of Olariu and Ben Rebea
(see Theorem~\ref{theorem olariu} and Lemma~\ref{lemma Ben Rebea}, respectively),
we present proofs of the following two results: 

\begin{theorem}\label{noAlpha}
Let $Y$ be a graph.
If $\mathcal{G}$ denotes the class of all connected $(K_{1,3},Y)$-free graphs
that are distinct from an odd cycle, then the following statements are equivalent.
\begin{itemize}
\item[$(1)$] Every graph of $\mathcal{G}$ is perfect.
\item[$(2)$] Every graph of $\mathcal{G}$ is $\omega$-colourable.
\item[$(3)$] $Y$ is an induced subgraph of $P_4$ or of $Z_1$.
\end{itemize}
Furthermore, if $P_4$ and $Z_1$ are $Y$-free, then there exist infinitely many graphs
of $\mathcal{G}$ which are not $\omega$-colourable.
\end{theorem}

\begin{theorem}\label{alpha3}
Let $Y$ be a graph.
If $\mathcal{G}_3$ denotes the class of all connected $(K_{1,3},Y)$-free graphs that are distinct from an odd cycle and which have independence number at least $3$, then the following statements are equivalent.
\begin{itemize}
\item[$(1)$] Every graph of $\mathcal{G}_3$ is perfect.
\item[$(2)$] Every graph of $\mathcal{G}_3$ is $\omega$-colourable.
\item[$(3)$] $Y$ is an induced subgraph of $P_5$ or of $Z_2$.
\end{itemize}
Furthermore, if $P_5$ and $Z_2$ are $Y$-free, then there exist infinitely many graphs
of $\mathcal{G}_3$ which are not $\omega$-colourable.
\end{theorem}

As a first step for proving these two theorems, we specify possible graphs $Y$ in Section~\ref{s2}.
The proofs of Theorems~\ref{noAlpha} and~\ref{alpha3} are presented in Section~\ref{s3}.
In addition, we study imperfect $(\claw,B)$-free graphs of independence number at least $3$ in Section~\ref{s4}. 

We conclude this section by recalling some further notation.
Given two graphs $G$ and $H$, we let $G \cup H$ denote the disjoint union of graphs $G$ and $H$ but write $2G$ for $G \cup G$.
For brevity, we let $[k]$ denote the set $\{1,2,\ldots,k\}$, and let $P_{k}$ ($C_{k}$) denote the path (cycle) on $k$ vertices. 
We use the notation $C_k\colon u_1u_2\dots u_{k}u_1$
to indicate the ordering of the vertices of the cycle $C_k$.
When working with the vertices of $C_k$, all calculations using indices are considered modulo~$k$.
Finally, let us remark that the graphs, that are depicted in Fig.\,\ref{figure forbsubgr},
are used throughout our paper.
\begin{center}
\begin{figure}[ht]
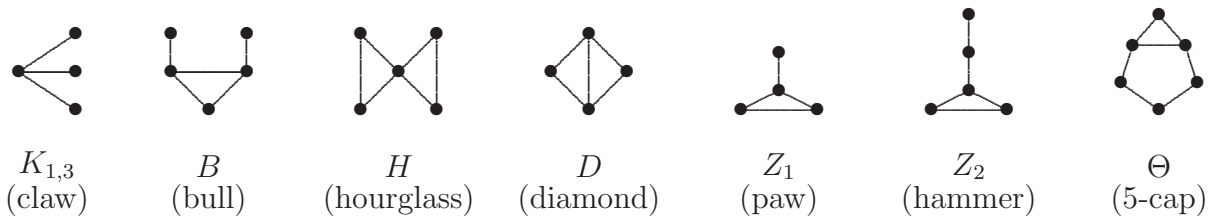

\beginpicture
\setcoordinatesystem units <0.55mm,0.5mm>
\setplotarea x from -5 to 180, y from -25 to 40
\put{
\beginpicture
\setcoordinatesystem units <0.5mm,0.5mm>
\setplotarea x from -5 to 5, y from  0 to 20
\put{$\bullet$} at 10 25
\put{$\bullet$} at 25 35
\put{$\bullet$} at 25 25
\put{$\bullet$} at 25 15
\plot 25 35  10 25  25 25  10 25  25 15 /
\put{$\bullet$} at 60 15
\put{$\bullet$} at 50 25
\put{$\bullet$} at 70 25
\put{$\bullet$} at 50 35
\put{$\bullet$} at 70 35
\plot 50 35  50 25  60 15  70 25  50 25  70 25  70 35 /
\put{$\bullet$} at 100 15
\put{$\bullet$} at 100 35
\put{$\bullet$} at 110 25
\put{$\bullet$} at 120 35
\put{$\bullet$} at 120 15
\plot 110 25  100 15  100 35  120 15  120 35  110 25 /
\put{$\bullet$} at 150 25
\put{$\bullet$} at 160 35
\put{$\bullet$} at 160 15
\put{$\bullet$} at 170 25
\plot 160 35  150 25  160 15  170 25  160 35  160 15 /
\put{$\bullet$} at 200 15
\put{$\bullet$} at 220 15
\put{$\bullet$} at 210 20
\put{$\bullet$} at 210 30
\plot 210 20  200 15  220 15  210 20 210 30 /
\put{$\bullet$} at 250 15
\put{$\bullet$} at 270 15
\put{$\bullet$} at 260 20
\put{$\bullet$} at 260 30
\put{$\bullet$} at 260 40
\plot 260 20  250 15  270 15  260 20  260 40 /
\put{$\claw$} at 17.5 0
\put{(claw)} at 17.5 -9
\put{$B$} at 60 0
\put{(bull)} at 60 -9
\put{$H$} at 110 0
\put{(hourglass)} at 110 -9
\put{$D$} at 160 0
\put{(diamond)} at 160 -9
\put{$Z_1$} at 210 0
\put{(paw)} at 210 -9
\put{$Z_2$} at 260 0
\put{(hammer)} at 260 -9
\put{$\bullet$} at  303 32   
\put{$\bullet$} at  317 32   
\put{$\bullet$} at  300 22  
\put{$\bullet$} at  320 22  
\put{$\bullet$} at  310 15  
\put{$\bullet$} at  310 40  
\plot 300 22  303 32  317 32  320 22  310 15  300 22  /
\plot 303 32  310 40  317 32 /
\put{$\Theta$} at 310 0
\put{($5$-cap)} at 310 -9

\endpicture} at  0 0

%
%
\endpicture
\caption{Some of the used forbidden subgraphs.}
\label{figure forbsubgr}
\end{figure}
\end{center}

\section{Auxiliary families of graphs}\label{s2}

In this section, we show the following:

\begin{lemma}\label{families}
For every graph $Y$ for which $P_4$ and $Z_1$ are $Y$-free,
there exist infinitely many connected $(K_{1,3},Y)$-free graphs that are distinct from an odd cycle but not $\omega$-colourable.
Furthermore, if $P_5$ and $Z_2$ are $Y$-free,
then there are infinitely many such graphs whose independence number is at least $3$.
\end{lemma}

We define several graph families which will be used for proving Lemma~\ref{families}.

For $s\geq 1$, we let $F_s^0$ be the graph obtained from the cycle $C_5$ by adding a complete graph $K_s$,
and by adding the edges connecting every vertex of 
the cycle to every vertex of the complete graph.
(To be more precise, the complement of $F_s^0$
is $sK_1 \cup C_5$.)

For $s\geq 3$, we let $F_s^1$ be the graph obtained from $C_{2s+1}\colon u_1u_2\dots u_{2s+1}u_1$ by adding $s$ isolated vertices $x_1,x_2, \dots, x_s$, and by adding the edges $x_iu_{2i-1}$, $x_iu_{2i}$ and $x_iu_{2i+1}$ for every $i\in[s]$.

For $s\geq 2$, we let $F_s^2$ be the graph obtained from the cycles
$C_{2s+1}\colon u_1u_2\dots u_{2s+1}u_1$
and $C_{5}\colon x_1x_2\dots x_{5}x_1$
by identifying the vertex $u_2$ with $x_2$ and identifying $u_3$ with $x_3$,
by adding a vertex $z$ such that $N(z) = \{x_1, x_2, \dots, x_5\}$,
and by adding the edges $x_1u_1$ and $x_4u_4$.

For $s\geq 1$, we let $F_s^3$ be the graph obtained from
$C_{6s+1}\colon u_1 u_2 \dots  u_{6s+1} u_1$
by adding $3s$ isolated vertices $x_1^i,x_2^i,x_3^i$ (for $i\in[s]$),
and by adding the edges $x_1^i u_{6i-5}$, $x_1^i u_{6i-4}$, $x_1^i u_{6i-2}$, $x_1^i u_{6i-1}$
and $x_2^i u_{6i-4}$, $x_2^i u_{6i-3}$, $x_2^i u_{6i-1}$, $x_2^i u_{6i}$,
and $x_3^i u_{6i-3}$, $x_3^i u_{6i-2}$, $x_3^i u_{6i}$, $x_3^i u_{6i+1}$ for every $i\in[s]$. 

We define graphs $F_s^4$ by considering their complements $\overline{F_s^4}$.
For odd $s\geq 3$, we let $\overline{F_s^4}$ be the graph obtained from
the cycles $C_{3s}\colon u_1u_2\dots u_{3s}u_1$  and $C_{3}\colon x_1 x_2 x_3 x_1$,
by adding the edges $x_1 u_{3i-2}$, $x_2 u_{3i-1}$ and $x_3 u_{3i}$ for every $i\in[s]$.

The structure of the graphs $F_s^1$, $F_s^2$, $F_s^3$ and $\overline{F_s^4}$
is depicted in Fig.\,\ref{figure 4graphs}.
In this section, we will use the above defined labelling of the vertices of $F_s^1$, $F_s^2$, $F_s^3$ and $F_s^4$.

\begin{figure}[ht]
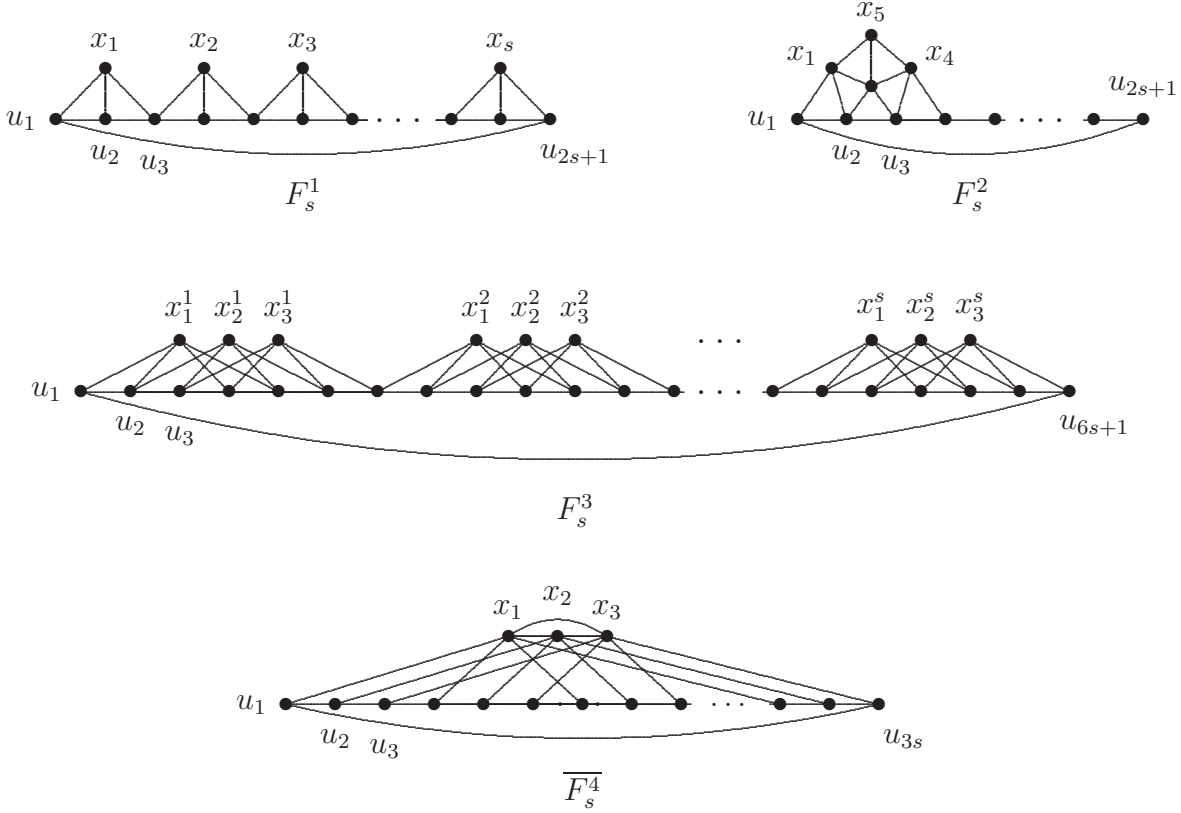

\beginpicture
\setcoordinatesystem units <0.65mm,0.45mm>
\setplotarea x from 0 to 155, y from -150 to 90
\put{$\bullet$} at 10 60
\put{$\bullet$} at 20 60
\put{$\bullet$} at 30 60
\put{$\bullet$} at 40 60
\put{$\bullet$} at 50 60
\put{$\bullet$} at 60 60
\put{$\bullet$} at 70 60
\put{$\bullet$} at 90 60
\put{$\bullet$} at 100 60
\put{$\bullet$} at 110 60
\put{$\bullet$} at 20 75
\put{$\bullet$} at 40 75
\put{$\bullet$} at 60 75
\put{$\bullet$} at 100 75
\put{$x_1$} at 20 83
\put{$x_2$} at 40 83
\put{$x_3$} at 60 83
\put{$x_s$} at 100 83
\put{$u_1$} at 3 60
\put{$u_2$} at 20 50
\put{$u_3$} at 30 47
\put{$u_{2s+1}$} at 115 50
\Large
\put{\dots} at 80 60
\normalsize
\plot 10 60  73 60  70 60  60 75  60 60  60 75  50 60  40 75  40 60  40 75  30 60  20 75  20 60  20 75  10 60 /
\plot 110 60  100 75  100 60  100 75  90 60  87 60  110 60 /
\setquadratic
\plot 10 60  60 50  110 60 /
\setlinear
\put{$\bullet$} at 170 60
\put{$\bullet$} at 180 60
\put{$\bullet$} at 190 60
\put{$\bullet$} at 200 60
\put{$\bullet$} at 220 60
\put{$\bullet$} at 230 60
\put{$\bullet$} at 160 60
\put{$\bullet$} at 175 70
\put{$\bullet$} at 167 75
\put{$\bullet$} at 183 75
\put{$\bullet$} at 175 85
\put{$u_1$} at 153 60
\put{$u_2$} at 170 50
\put{$u_3$} at 180 47
\put{$u_{2s+1}$} at 230 70
\put{$x_1$} at 161 79
\put{$x_4$} at 189 79
\put{$x_5$} at 175 92
\Large
\put{\dots} at 210 60
\normalsize
\plot 202 60  160 60  167 75  175 85  183 75  190 60  180 60  183 75  175 70  175 85  175 70  170 60  167 75  175 70  180 60 /
\plot 218 60  230 60 /
\setquadratic
\plot 160 60  195 50  230 60 /
\setlinear
\put{
\beginpicture
\put{$\bullet$} at 20 0
\put{$\bullet$} at 30 0
\put{$\bullet$} at 40 0
\put{$\bullet$} at 50 0
\put{$\bullet$} at 60 0
\put{$\bullet$} at 70 0
\put{$\bullet$} at 80 0
\put{$\bullet$} at 90 0
\put{$\bullet$} at 100 0
\put{$\bullet$} at 110 0
\put{$\bullet$} at 120 0
\put{$\bullet$} at 130 0
\put{$\bullet$} at 140 0
\put{$\bullet$} at 160 0
\put{$\bullet$} at 170 0
\put{$\bullet$} at 180 0
\put{$\bullet$} at 190 0
\put{$\bullet$} at 200 0
\put{$\bullet$} at 210 0
\put{$\bullet$} at 220 0
\put{$\bullet$} at 40 15
\put{$\bullet$} at 50 15
\put{$\bullet$} at 60 15
\put{$\bullet$} at 100 15
\put{$\bullet$} at 110 15
\put{$\bullet$} at 120 15
\put{$\bullet$} at 180 15
\put{$\bullet$} at 190 15
\put{$\bullet$} at 200 15
\plot 142 0  20 0  40 15  30 0  50 0  40 15  60 0  80 0  60  15  70 0  50 15  60 0  50 0  60 15  40 0  50 15  30 0  80 0  100 15  90 0  110 15  100 0  120 15  110 0  100 15  120 0 110 15  130 0  120 15  140 0 /
\plot 158 0  220 0  200 15  210 0  190 15  200 0  180 15  190 0  200 15  180 0  190 15  170 0  180 15  160 0 /
\Large
\put{\dots} at 150 0
\put{\dots} at 150 15
\normalsize
\put{$x_1^1$} at 40 25
\put{$x_2^1$} at 50 25
\put{$x_3^1$} at 60 25
\put{$x_1^2$} at 100 25
\put{$x_2^2$} at 110 25
\put{$x_3^2$} at 120 25
\put{$x_1^s$} at 180 25
\put{$x_2^s$} at 190 25
\put{$x_3^s$} at 200 25
\put{$u_1$} at 13 0
\put{$u_2$} at 30 -10
\put{$u_3$} at 40 -13
\put{$u_{6s+1}$} at 225 -10
\setquadratic
\plot 20 0  120 -20  220 0 /
\setlinear
\put{$F_s^3$} at 120 -35
\endpicture
} at 115 -25

\put{
\beginpicture
\put{$\bullet$} at 20 0
\put{$\bullet$} at 30 0
\put{$\bullet$} at 40 0
\put{$\bullet$} at 50 0
\put{$\bullet$} at 60 0
\put{$\bullet$} at 70 0
\put{$\bullet$} at 80 0
\put{$\bullet$} at 90 0
\put{$\bullet$} at 100 0
\put{\dots} at 110 0
\put{$\bullet$} at 120 0
\put{$\bullet$} at 130 0
\put{$\bullet$} at 140 0
\put{$\bullet$} at 65 20
\put{$\bullet$} at 75 20
\put{$\bullet$} at 85 20
\plot 102 0  20 0  65 20  50 0  80 0  65 20  85 20  65 20  120 0 /
\plot 30 0  75 20  60 0  90 0  75 20  130 0 /
\plot 40 0  85 20  70 0  100 0  85 20  140 0  118 0 /
\setquadratic
\plot 20 0  80 -10  140 0 /
\plot 65 20  75 25  85 20 /
\setlinear
\put{$\overline{F_s^4}$} at 80 -25
\put{$u_1$} at 13 0
\put{$u_2$} at 30 -10
\put{$u_3$} at 40 -13
\put{$u_{3s}$} at 145 -10
\put{$x_1$} at 65 28
\put{$x_2$} at 75 32
\put{$x_3$} at 85 28
\Large
\put{\dots} at 80 0
\normalsize
\endpicture
} at 115 -110

\put{$F_s^1$} at 60 38
\put{$F_s^2$} at 195 38

\endpicture
\caption{The structure of the graphs $F_s^1$, $F_s^2$, $F_s^3$ and $\overline{F_s^4}$.}
\label{figure 4graphs}
\end{figure}

We let $\mf^0$ denote the family which consists of all graphs $F_s^0$.
Similarly, we let $\mf^1$, $\mf^2$, $\mf^3$, $\mf^4$ denote
the family of all graphs $F_s^1$, $F_s^2$, $F_s^3$, $F_s^4$, respectively.
We show the following:

\begin{observ} \label{obs col}
None of the graphs of $\mathcal{F}^0 \cup \mf^1 \cup \mf^2 \cup \mf^3 \cup \mf^4$ is $\omega$-colourable.
\end{observ}

\begin{proof}
For the sake of a contradiction, we suppose that $F\in \mathcal{F}^0 \cup \mf^1 \cup \mf^2 \cup \mf^3 \cup \mf^4$ is $\omega$-colourable.

If $F\in \mathcal{F}^0$, then it can be easily seen that $\omega(F)=s+2$ and $\chi(F)=s+3$, a contradiction.

If $F\in\mf^1$, then clearly $\omega(F)=3$. Since, by our supposition, $F$ is $3$-colourable, we consider a $3$-colouring of the subgraph induced by $\{ u_1, u_2, u_3, x_1 \}$,
and note that the vertices $u_1$ and $u_3$ must have the same colour.
We apply a similar reasoning for $u_3$ and $u_5$, for $u_5$ and $u_7$, \ldots, and so on.  In particular, we note that $u_1$ and $u_{2s+1}$ have the same colour, which is a contradiction since both vertices are adjacent.

If $F\in \mf^2$, then clearly $\omega(F)=3$. But since $F_1^0$ is an induced subgraph of $F$, we have $\chi(F)\geq\chi(F_1^0)>\omega(F_1^0)=3$, a contradiction.

If $F\in\mf^3$, then we observe $\omega(F)=3$. Since, by our supposition, $F$ is $3$-colourable, we consider a $3$-colouring, say $\chi$, of the subgraph induced by $\{ u_1, u_2, \dots, u_7, x_1^1, x_2^1, x_3^1 \}$.
We prove that the vertices $u_1$ and $u_7$ must have the same colour.
Since $\{u_1,u_2,x_1^1\}$ induces $K_3$, we assume $\chi(u_1)=a$, $\chi(u_2)=b$, and $\chi(x_1^1)=c$ for discussing the two possible cases.
\begin{itemize}
\item[$\bullet$] If $\chi(u_3)=a$, then $\chi(u_4)=b$ and $\chi(x_2^1)=c$.
This implies $\chi(u_5)=a$ and $\chi(x_3^1)=c$.
Consequently, we get $\chi(u_6)=b$, and thus $\chi(u_7)=a$.
\item[$\bullet$] If $\chi(u_3)=c$, then $\chi(x_2^1)=a$, and therefore $\chi(u_5)=b$.
Hence, $\chi(u_4)=a$ and $\chi(u_6)=c$,
and we get $\chi(x_3^1)=b$ and $\chi(u_7)=a$.
\end{itemize}
By a similar argument, we note that $\chi(u_{6i+1})=a$ for every $i\in[s]$, contradicting the fact that $u_1$ and $u_{6s+1}$ are adjacent.

If $F\in \mf^4$, then let $F\cong F_s^4$ for some odd $s\geq 3$. 
%
For showing $\chi(F) >\omega(F)$, we consider an arbitrary independent set, say $I$, of $\overline{F}$.
In the case where a vertex of $\{ x_1, x_2, x_3 \}$ belongs to $I$,
we note that $|I| \leq s+1$.
Otherwise, the set $I$ may contain up to $\frac{3s-1}{2}$ vertices since $V(F) \setminus \{ x_1, x_2, x_3 \}$ induces an odd cycle on $3s$ vertices in $\overline{F}$.
Clearly, $\frac{3s-1}{2} \geq s + 1$ since $s \geq 3$, and thus $\omega(F)=\frac{3s-1}{2}$.
Since $F$ is $4K_1$-free and $\{ x_1, x_2, x_3 \}$ is the only independent set of size~$3$ in $F$, 
at most one colour can be used three times in a colouring assigning distinct colours to adjacent vertices. So, $\chi(F)\geq 1+\frac{|V(F)|-3}{2}$.
By recalling that $|V(F)| = 3s +3$ and $|V(F)|-3$ is odd, we have $\chi(F) \geq \frac{3s+3}{2}>\frac{3s-1}{2}=\omega(G)$, a contradiction.
\end{proof}

We apply Observation~\ref{obs col} and prove Lemma~\ref{families}.

\begin{proof}[Proof of Lemma~\ref{families}]
We recall the families of graphs $\mathcal{F}^0,\mf^1,\mf^2,\mf^3$ and $\mf^3$ as defined above. 
We note that every graph of
$\mathcal{F}^0\cup\mathcal{F}^1\cup\mathcal{F}^2\cup\mathcal{F}^3\cup\mathcal{F}^4$
is connected and distinct from an odd cycle.
Furthermore, every such graph is not $\omega$-colourable by Observation~\ref{obs col}.

We note that all graphs of $\mathcal{F}^0$ are $(3K_1, 2K_2, K_1 \cup K_3)$-free
(and thus $\claw$-free).
Furthermore, it is not hard to see that all graphs of $\mathcal{F}^1$ and all graphs of $\mathcal{F}^3$ are
$(\claw, K_4, C_4, C_5)$-free and
$(\claw, D)$-free, respectively.
Hence, if $Y$ does not belong to the class of $(3K_1, 2K_2, K_1\cup K_3,$ $K_4, D, C_4, C_5)$-free graphs,
then there are infinitely many connected $(K_{1,3},Y)$-free graphs that are distinct from an odd cycle but not $\omega$-colourable.

Consequently, we can assume that $Y$ is
$(3K_1, 2K_2, K_1 \cup K_3, K_4, D, C_4, C_5)$-free.
We show that $Y$ is an induced subgraph of $P_4$ or of $Z_1$.
Since $Y$ is $(3K_1,C_4,C_5)$-free, it is $(K_{1,3}, C_4, C_5, C_6, \ldots)$-free as well. 
In particular, if $Y$ is $K_3$-free, then every component of $Y$ is a path,
and thus $Y$ is an induced subgraph of $P_4$ since $Y$ is $(3K_1, 2K_2)$-free.

We assume that $Y$ contains a set, say $T$, of vertices that induces a complete subgraph of order~$3$.
Since $Y$ is $(K_1 \cup K_3, K_4, D)$-free, every vertex outside $T$ has precisely one neighbour in $T$.
Furthermore, every edge of $Y$ is incident with a vertex of $T$ since $Y$ is $(2K_2, C_4)$-free.
Consequently, we note that $Y$ has at most one vertex outside $T$ since $Y$ is $3K_1$-free.
Thus, $Y$ is an induced subgraph of $Z_1$.


In order to show the ``furthermore part" of the statement,
we consider the families of graphs $\mf^1$, $\mf^2$, $\mf^3$ and $\mf^4$.
We note that the families consist of graphs which all have independence number at least $3$ 
(for graphs of $\mf^1$ and $\mf^3$, we take $\{u_1,u_3,u_5\}$ as an independent set;
for graphs of $\mf^2$, we take $\{u_1,u_3,x_5\}$ as an independent set;
and for $\mf^4$, we take $\{x_1,x_2,x_3\}$ as an independent set.).

Furthermore, we observe that all graphs of $\mf^1$, all graphs of $\mf^2$, and all graphs of $\mf^3$, 
are $(\claw, B, K_4, C_4, C_5, C_6)$-free,
$(\claw, H)$-free, and
$(\claw, D)$-free, respectively.
In addition, all graphs of $\mf^4$ are $(\claw, 4K_1, 2K_1 \cup K_2, K_2 \cup K_3)$-free since their complements are $(K_1 \cup K_3, K_4, D, K_{2,3})$-free.

Similarly as above, we assume that $Y$ is
$(\claw, 4K_1, 2K_1 \cup K_2, K_2 \cup K_3, B, K_4, D, H$, $C_4, C_5, C_6)$-free,
and we show that $Y$ is an induced subgraph of $P_5$ or of $Z_2$.
We note that
$Y$ is $(K_{1,3}, C_4, C_5, C_6, C_7, \ldots)$-free
since it is $(K_{1,3}, C_4, C_5, C_6, 2K_1 \cup K_2)$-free. 
In particular, if~$Y$ is $K_3$-free, then every component of $Y$ is a path,
and $Y$ is an induced subgraph of $P_5$ since $Y$ is $(4K_1, 2K_1 \cup K_2)$-free.

We assume that there is a set, say $T$, of vertices that induces a complete subgraph of order~$3$.
We note that every vertex outside $T$ has at most one neighbour in $T$ since $Y$ is $(K_4, D)$-free.
Furthermore, at most one vertex outside $T$ is adjacent to a vertex of $T$ since $Y$ is $(2K_1 \cup K_2, B, H, C_4)$-free,
and at most one vertex outside $T$ has no neighbour in $T$ since $Y$ is $(2K_1 \cup K_2, K_2 \cup K_3)$-free.
In~case $Y$ has two vertices outside $T$,
we note that these two vertices are adjacent since $Y$ is $2K_1 \cup K_2$-free.
Thus, we conclude that $Y$ is an induced subgraph of $Z_2$. 
\end{proof}

\section{Proving the main results}\label{s3}
In this section, we apply Theorem~\ref{SPGT} and Lemma~\ref{families}
and prove our main results, Theorem~\ref{noAlpha} and Theorem~\ref{alpha3}.
For proving Theorem~\ref{noAlpha}, we shall also use the following:

\begin{theorem}[Olariu~\cite{O}] \label{theorem olariu}
If $G$ is a connected $Z_1$-free graph, then $G$ is $K_3$-free or complete multipartite.
\end{theorem}

\begin{proof}[Proof of Theorem~\ref{noAlpha}]
By definition, every perfect graph is $\omega$-colourable.
So, $(1)$ implies $(2)$.

We note that the ``furthermore part" of the statement follows by Lemma~\ref{families}.
In particular, we deduce from this lemma that $(2)$ implies $(3)$.

We show that $(3)$ implies $(1)$.
Clearly, if $Y$ is an induced subgraph of $P_4$,
then the implication follows from Theorem~\ref{SPGT} since every $P_4$-free graph contain neither an odd hole
nor an odd antihole.
If $Y$ is an induced subgraph of $Z_1$, then, by Theorem~\ref{theorem olariu}, every $Z_1$-free graph is $K_3$-free or complete multipartite.
We note that all complete multipartite graphs are perfect. Furthermore, every $(K_{1,3},K_3)$-free graph has maximum degree at most $2$. 
Thus, we conclude that every $K_3$-free graph of $\mathcal{G}$ is perfect as well since it is connected, distinct from an odd cycle, and has maximum degree at most $2$.
\end{proof}

For further reference, we note the following:

\begin{observ}
\label{N(x)capC}
Let $G$ be a $\claw$-free graph and let $C$ be a set of its vertices
such that $C$ induces a cycle of length at least $5$.
If $x$ is a vertex that does not belong to $C$ but is adjacent to a vertex of $C$, then $N(x) \cap C$ induces $K_2$ or $P_3$ or $P_4$
or a $C_5$
or a $2K_2$.  
\end{observ}

\begin{proof}
Since $G$ is $\claw$-free, $N(x)$ cannot contain three pairwise independent vertices.
Furthermore, the graph induced by $N(x) \cap C$ cannot contain $K_1$ as a component.
\end{proof}

In addition,
we shall also use the following two Lemmas on $\claw$-free graphs for proving Theorem~\ref{alpha3}.
The first one is due to Ben Rebea.

\begin{lemma}[Ben Rebea's Lemma~\cite{BR,Chv}\label{lemma Ben Rebea}]
Let $G$ be a connected $\claw$-free graph with independence number at least $3$.
If $G$ contains an odd antihole, then it contains a hole of order~$5$.
\end{lemma}

Let $\Theta$ be the graph obtained from $C_6$ by adding one edge between two vertices of distance $2$ in $C_6$ (see Fig.\,\ref{figure forbsubgr}).

\begin{lemma}\label{C_5}
Every connected $(K_{1,3}, \Theta )$-free graph with independence number at least $3$
is $C_5$-free.
\end{lemma}

\begin{proof}
We let $G$ be a connected $(K_{1,3}, \Theta )$-free graph with independence number at least $3$.
For the sake of a contradiction, we suppose that $G$ contains a set $C$ of vertices inducing a cycle of length $5$. Since $G$ is connected and contains an independent set of size at least $3$, there exists a vertex which does not belong to $C$ but is adjacent to a vertex of $C$. Let $x$ be an arbitrary one.
We note that the graph induced by $N(x) \cap C$ is distinct from $K_2$ since $G$ is $\Theta$-free,
and thus, by Observation~\ref{N(x)capC}, it is either $P_3$ or $P_4$ or $C_5$.
In particular, $x$ has two non-adjacent neighbours in $C$,
and therefore a vertex which has no neighbour in $C$
cannot be adjacent to $x$ 
since $G$ is $K_{1,3}$-free.
Since $G$ is connected,
it follows from the arbitrariness of $x$ that every vertex of $G$ is adjacent to a vertex of $C$.

We let $I$ be a set of $3$ independent vertices of $G$.
We consider the graph induced by $C \cup I$,
and note that no vertex of $C$ is adjacent to all vertices of $I$ since $G$ is $\claw$-free.
Thus, there are only three options of interconnecting $C$ and $I$
(see Fig.\,\ref{fig-CcupI}).
%
%
%
%
%
\begin{figure}[ht]
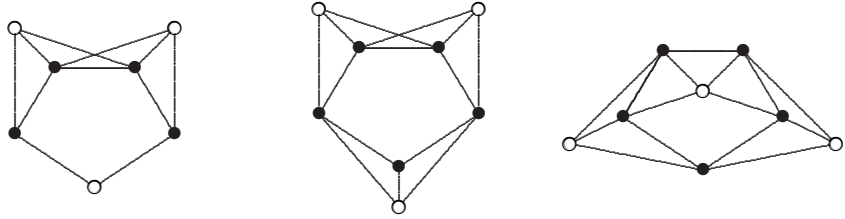

$$\beginpicture
\setcoordinatesystem units <0.8mm,0.8mm>
\setplotarea x from -50 to 50, y from -2 to 2
\put{
\beginpicture
\setcoordinatesystem units <0.35mm,0.35mm>
\setplotarea x from -5 to 5, y from  0 to 20
\put{$\bullet$} at   -30   30
\put{$\bullet$} at    30   30
\put{$\bullet$} at   -15   15
\put{$\bullet$} at    15   15    
\put{$\bullet$} at  -30   -10
\put{$\bullet$} at   30   -10
\put{$\bullet$} at   0	  -30
\plot -30 -10  -30 30  -15 15 15 15 30 30 30 -10 0 -30 -30 -10 -15 15 15 15 -30 30 /
\plot -15 15  30 30 /
\plot  15 15  30 -10 /
\put{$\wv$} at   -30    30
\put{$\wv$} at   30    30
\put{$\wv$} at   0    -30
\endpicture}  at  -50  0
\put{
\beginpicture
\setcoordinatesystem units <0.35mm,0.35mm>
\setplotarea x from -5 to 5, y from  0 to 20
\put{$\bullet$} at   -30   30
\put{$\bullet$} at    30   30
\put{$\bullet$} at   -15   15
\put{$\bullet$} at    15   15    
\put{$\bullet$} at  -30   -10
\put{$\bullet$} at   30   -10
\put{$\bullet$} at   0	  -30
\put{$\bullet$} at   0	  -45
\plot -30 -10  -30 30  -15 15 15 15 30 30 30 -10 0 -30 -30 -10 -15 15 15 15 -30 30 /
\plot -15 15  30 30 /
\plot  15 15  30 -10 /
\plot  30 -10 0 -45  -30 -10 /
\plot  0 -30  0 -45 /
\put{$\wv$} at   -30    30
\put{$\wv$} at   30    30
\put{$\wv$} at   0    -45
\endpicture}  at    0  0
\put{
\beginpicture
\setcoordinatesystem units <0.35mm,0.35mm>
\setplotarea x from -5 to 5, y from  0 to 20
\put{$\bullet$} at   -50   -20
\put{$\bullet$} at    50   -20
\put{$\bullet$} at   -15   15
\put{$\bullet$} at    15   15    
\put{$\bullet$} at  -30   -10
\put{$\bullet$} at   30   -10
\put{$\bullet$} at   0	  -30
\put{$\bullet$} at   0	   0
\plot -30 -10  -15 15 15 15 30 -10 0 -30 -30 -10 -15 15 15 15 /
\plot -15 15  -50 -20 0 -30 50 -20 15 15 /
\plot  -50 -20 -30 -10 0 0 30 -10 50 -20 /
\plot -15 15 0 0 15 15 /
\put{$\wv$} at   0    0
\put{$\wv$} at   -50    -20
\put{$\wv$} at   50    -20
\endpicture}  at   50  0
\endpicture$$
\caption{The only possible realisations of the graph induced by $C \cup I$, where the white vertices depict the vertices of $I$.}
\label{fig-CcupI}
\end{figure}
We conclude that the graph induced by $C \cup I$ is not $\Theta$-free,
a contradiction.
\end{proof}

We combine the ingredients and finally prove Theorem~\ref{alpha3}.

\begin{proof}[Proof of Theorem~\ref{alpha3}]
Similarly as in the proof of Theorem~\ref{noAlpha},
we note that $(1)$ implies $(2)$, and $(2)$ implies $(3)$,
and that the ``furthermore part" is satisfied.

We show that $(3)$ implies $(1)$.
We consider a graph $G$ of $\mathcal{G}_3$, and proceed in two steps.

First of all, we show that $G$ is $(C_7, C_9, \dots )$-free.
Clearly, the claim is satisfied in case $Y$ is an induced subgraph of $P_5$.
We assume that $Y$ is an induced subgraph of $Z_2$.
For the sake of a contradiction, we suppose that $G$ contains a set $C$ of vertices inducing an
odd cycle of length at least $7$.
Since $G$ is connected and distinct from an odd cycle,
there is a vertex $x$ that does not belong to $C$ but is adjacent to a vertex of $C$.
By Observation~\ref{N(x)capC}, the graph induced by $N(x) \cap C$ is either 
$K_2$ or $P_3$ or $P_4$ or $2K_2$.
In all cases, the graph induced by $C \cup \{ x \}$ is not $Z_2$-free,
a contradiction.

For our second step, we note that $\Theta$ is neither $P_5$-free nor $Z_2$-free. Hence $G$ 
is $(\claw, \Theta)$-free,
and thus $G$ is $C_5$-free by Lemma~\ref{C_5}.
In particular, $G$ contains no odd antihole by Lemma~\ref{lemma Ben Rebea}, which implies that $G$ is perfect by Theorem~\ref{SPGT}.
\end{proof}

\section{Characterising exceptional graphs}\label{s4}

For every graph $Y$, we considered connected $(\claw, Y)$-free graphs that are distinct from an odd cycle but which have an independence number of at least $3$, and we showed that either all such graphs
are perfect or there are infinitely many such graphs which are not $\omega$-colourable.
We note that for some graphs $Y$, not $\omega$-colourable $(\claw, Y)$-free graphs that are distinct from an odd cycle but which have an independence number of at least $3$ could be described precisely.

In this section, we study $(\claw, B)$-free graphs but need the terminology of an inflation of a cycle. For $k\geq 4$ and $n_1,n_2,\ldots,n_k\geq 1$,
we let $C[n_1,n_2,\ldots,n_k]$ denote the graph whose vertex set can be partitioned into
$k$ sets $W_1,W_2,\ldots,W_k$
of sizes $n_1,n_2,\ldots,n_k$, respectively,
such that
\begin{itemize}
\item[$\bullet$]
$W_i\cup W_j$ induces a complete graph (of order $n_i+n_j$)
if either $|i - j| = 1$ or $|i - j| = k - 1$,
\item[$\bullet$]
 $W_i\cup W_j$ induces a disjoint union of two complete graphs if $1<|i - j|<k-1$
\end{itemize} 
for all distinct $i,j\in [k]$.
For brevity, we say that the graph $C[n_1,n_2,\ldots,n_k]$ is an {\it inflation of the cycle} $C_k$.
Two examples are depicted in Fig.\,\ref{fig2-class_C}.

\begin{figure}[ht]
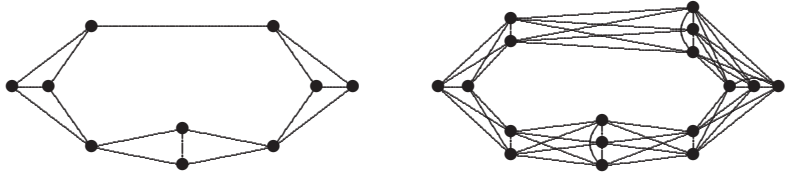

$$\beginpicture
\setcoordinatesystem units <0.8mm,1.0mm>
\setplotarea x from -60 to 60, y from -10 to 10
\put
 {\beginpicture
 \setcoordinatesystem units <0.8mm,0.8mm>
 \setplotarea x from -20 to 20, y from -15 to 15
 \put{$\bullet$} at  -15    10
 \put{$\bullet$} at   15    10  
 \put{$\bullet$} at   22     0  
 \put{$\bullet$} at   28     0  
 \put{$\bullet$} at   15   -10 
 \put{$\bullet$} at    0    -7
 \put{$\bullet$} at    0   -13
 \put{$\bullet$} at  -15   -10
 \put{$\bullet$} at  -22     0    
 \put{$\bullet$} at  -28     0    
 \plot -15 10  15 10  28 0  15 -10  0 -13  -15 -10 -28 0  -15 10 /
 \plot  15 10  22 0  15 -10  0 -7  -15 -10 -22 0  -15 10 /
 \plot -28 0  -22  0 /
 \plot  28 0   22  0 /
 \plot 0 -7  0 -13 /
\endpicture} at  -35  0
\put
 {\beginpicture
 \setcoordinatesystem units <0.8mm,0.75mm>
 \setplotarea x from -20 to 20, y from -15 to 15
 \put{$\bullet$} at  -15     8
 \put{$\bullet$} at  -15    12
 \put{$\bullet$} at   15     6
 \put{$\bullet$} at   15    10
 \put{$\bullet$} at   15    14
 \put{$\bullet$} at   21     0  
 \put{$\bullet$} at   25     0  
 \put{$\bullet$} at   29     0  
 \put{$\bullet$} at   15    -8 
 \put{$\bullet$} at   15   -12
 \put{$\bullet$} at    0    -6
 \put{$\bullet$} at    0   -10
 \put{$\bullet$} at    0   -14
 \put{$\bullet$} at  -15    -8
 \put{$\bullet$} at  -15   -12
 \put{$\bullet$} at  -22     0    
 \put{$\bullet$} at  -27     0 
 \plot -27 0  -15 12  -22 0  -15 8  -27 0  -22 0 /  
 \plot -27 0  -15 -12  -22 0  -15 -8  -27 0 /  
 \plot -15 -12  0 -6  -15 -8  0 -10 -15 -12  0 -14  -15 -8 /
 \plot  15 -12  0 -6   15 -8  0 -10  15 -12  0 -14   15 -8 /
 \plot 15 -12  21 0  15 -8  25 0  15 -12  29 0  15 -8 /
 \plot 15 10  21 0  15 6  25 0  15 10  29 0  15 14  25 0 /
 \plot 21 0  15 14 /
 \plot 29 0  15 6 /
 \plot -15 8  15 14  -15 12  15 10  -15 8  15 6  -15 12 /
 \plot -15 8  -15 12 /
 \plot  15 -8  15 -12 /
 \plot  15 6   15 14 /
 \plot -15 -8  -15 -12 / 
 \plot 0 -6  0 -14 /
 \plot 21 0  29 0 /
 \setquadratic
 \plot 0 -6  -2 -10  0 -14 /
 \plot 21 0  25 1.5  29 0 /
 \plot 15 6  13 10  15 14 /
\endpicture} at   35  0
\endpicture$$
%
%
\caption{Two examples of an inflation of $C_7$.}
\label{fig2-class_C}
\end{figure}

Our main result of this section characterises connected imperfect $(\claw, B)$-free graphs with independence number at least $3$.

\begin{theorem}\label{perfect_bull}
Every connected $(\claw, B)$-free graph with independence number at least $3$
is either perfect or an inflation of an odd cycle of length at least $7$.
\end{theorem}

We note that Theorem~\ref{perfect_bull} follows
from the combination of Theorem~\ref{SPGT} and Lemma~\ref{lemma Ben Rebea}
and the following lemma.

\begin{lemma}\label{lemma_characterisation_bull}
Let $G$ be a connected $(\claw,B)$-free graph.
\begin{itemize}
\item[$(1)$] If $G$ contains an independent set of size $3$, then $G$ is $C_5$-free.
\item[$(2)$] If $G$ contains an induced cycle of length $k$ with $k\geq 6$, then
     $G$ is an inflation of $C_k$.
\end{itemize} 
\end{lemma}

\begin{proof}
We note that $G$ 
is $(\claw, \Theta)$-free since it is $(\claw, B)$-free,
and thus statement $(1)$ follows by Lemma~\ref{C_5}.

It remains to show $(2)$.
We let $C$ be a set of vertices inducing a cycle of length $k\geq 6$ in $G$.
Clearly, if $V(G) = C$, then the statement is trivially satisfied. 
Since $G$ is connected, we can assume that there is
a vertex which does not belong to $C$ but is adjacent to a vertex of $C$. Let $x$ be an arbitrary one.
We consider the graph induced by $N(x) \cap C$,
and note that this graph is distinct from $K_2$, $2K_2$, and $P_4$
since $G$ is $B$-free.
Thus, by Observation~\ref{N(x)capC}, the graph induced by $N(x) \cap C$ is a path of $3$ vertices, which implies that every vertex outside $C$ is adjacent to a vertex of $C$
since $G$ is connected but also $\claw$-free.
We let $C_k\colon v_1v_2\ldots v_kv_1$ be the cycle induced by $C$,
and divide $V(G)\setminus C$ into $k$ disjoint sets $M_1,M_2,\ldots ,M_k$ so that if
$w$ belongs to $M_i$, then $N(w) \cap C =\{v_i,v_{i+1},v_{i+2}\}$ for every $i\in [k]$.

Next, let us consider two vertices of $V(G)\setminus C$, say $w$ and $w'$.
We let $c := |N(w) \cap N(w') \cap C|$, that is, $c$ denotes the number of common neighbours of $w$ and $w'$ in $C$. We deduce the following:

\begin{itemize}
\item[$\bullet$] If $c=3$, then $w$ and $w'$ are adjacent
since $G$ is $\claw$-free.
\item[$\bullet$] If $c=2$, then $w$ and $w'$ are adjacent
since $G$ is $B$-free.

\item[$\bullet$] If $c=1$, then $w$ and $w'$ are non-adjacent
since $G$ is $B$-free.
\item[$\bullet$] If $c=0$, then $w$ and $w'$ are non-adjacent
since $G$ is $\claw$-free.
\end{itemize}
Consequently, we note that $G$ is an inflation of $C_k$. 
\end{proof}

\end{document}